\documentclass[11pt,a4]{amsart}
\usepackage[margin=2.5cm]{geometry}
\usepackage{oldgerm}
\usepackage{amssymb} 
\usepackage{mathrsfs} 
\usepackage{amsmath}
\usepackage{latexsym}
\usepackage[all]{xy}
\usepackage[dvips]{graphics}
\usepackage[dvipdfmx,bookmarks=true]{hyperref}
\usepackage{amsthm}
\usepackage{enumerate}

\allowdisplaybreaks

\theoremstyle{plain} 
\newtheorem{thm}{Theorem}[section] 
\newtheorem{prop}[thm]{Proposition} 
\newtheorem{lem}[thm]{Lemma} 
\theoremstyle{definition} 
\newtheorem{rem}[thm]{Remark} 
\newtheorem{fact}[thm]{Fact} 
\newtheorem*{acknow}{Acknowledgements} 
\theoremstyle{remark} 
\newtheorem*{notation}{Notation}
\def\bfchi{\boldsymbol{\chi}}

\title[A semi-ordinary $p$-stabilization of Siegel Eisenstein series]
{A semi-ordinary $p$-stabilization of 
Siegel Eisenstein series for symplectic groups and its $p$-adic interpolation
}

\author{Hisa-aki KAWAMURA}
\email{kawamura.hisaaki@hiro.kindai.ac.jp}

\date{February 23rd, 2023}


\begin{document}
\maketitle

\begin{abstract}
Given a prime number $p$, we introduce a certain $p$-stabilization of holomorphic Siegel Eisenstein series 
for the symplectic group ${\rm Sp}(2n)_{/\mathbb{Q}}$ such that the resulting forms 
satisfy the semi-ordinary condition at $p$, 
that is, the associated eigenvalue of a generalized Atkin $U_p$-operator is a $p$-adic unit. 
In addition, we derive an explicit formula for all Fourier coefficients of such $p$-stabilized Siegel Eisenstein series, and 
conclude their $p$-adic interpolation problems. 
This states the existence of a quite natural generalization of the ordinary $\Lambda$-adic Eisenstein series 
which have been constructed by Hida and Wiles for ${\rm GL(2)}_{/\mathbb{Q}}$. 
\end{abstract}


\section{Introduction}

\vspace*{2mm}
Given a positive integer $M$, a Dirichlet character $\chi$ modulo $M$ and an integer $\kappa \ge 2$ with $\chi(-1)=(-1)^{\kappa}$, 
the classical (holomorphic) Eisenstein series $E_{\kappa,\,\chi} \bigl(=E_{\kappa,\,\chi}^{(1)}(z)\bigr)$ is defined as follows: 
For each $z \in \mathfrak{H}_1=\{z \in \mathbb{C}\mid{\rm Im}\, z >0\}$, put
\[
E_{\kappa,\,\chi}(z)
:= 
\left\{{2\,G(\chi^{-1})(-2\pi\sqrt{-1})^\kappa \over M^\kappa (\kappa-1)!}\right\}^{-1}
\sum_{(c,d) \in \mathbb{Z}^2 \smallsetminus \{(0,0)\}} \chi^{-1}(d) \,(c\,Mz+d)^{-\kappa},
\] 
where $G(\chi^{-1})=
\sum_{n=1}^{M} \chi^{-1}(n) \exp(2\pi\sqrt{-1}\,n/M)$. 
As is well-known, $E_{\kappa,\,\chi}$ gives rise to a holomorphic modular form of weight $\kappa$ and 
nebentypus character $\chi$ for the congruence subgroup $\Gamma_0(M)$ of ${\rm SL}(2, \mathbb{Z})$ unless $\kappa=2$ and $\chi$ is trivial (or principal
). 
If $\chi$ is trivial (i.e., $M=1$) or primitive (i.e., the conductor of $\chi$ equals $M>1$), then $E_{\kappa,\,\chi}$ possesses the Fourier expansion 
\[
E_{\kappa,\,\chi}(z)={L(1-\kappa,\chi) \over 2}+ \sum_{m =1}^{\infty} \sigma_{\kappa-1,\, \chi}(m)\, q^m, 
\]
where 
$L(s, \chi)=
\prod_{\scriptstyle l \,{\rm : \,prime}} (1-\chi(l)\,l^{-s})^{-1}$, $\sigma_{\kappa-1,\,\chi}(m)=
\sum_{\scriptstyle 0<d \,\mid\, m} \chi(d)\,d^{\kappa-1}$ and $q=\exp(2\pi \sqrt{-1}z)$. 
This implies that 
$E_{\kappa,\,\chi}$ is a non-cuspidal Hecke eigenform such that the associated $L$-function $L(s,E_{\kappa,\,\chi})$ is taken of the form 
\[
L(s,E_{\kappa,\,\chi})=\zeta(s) L(s-\kappa+1,\chi). 
\] 

Let $p$ be a fixed prime number not dividing $M$, which we assume to be odd for simplicity. 
Put 
\begin{equation}
E_{\kappa,\,\chi}^{*}(z)
:=E_{\kappa,\,\chi}(z) - \chi(p)\,p^{\kappa-1} E_{\kappa,\,\chi}(pz).
\end{equation}
We easily see that 
$E_{\kappa,\,\chi}^{*}$ is also a non-cuspidal Hecke eigenform of weight $\kappa$ 
and nebentypus character $\chi$ for $\Gamma_0(Mp) \,(\subset \Gamma_0(M))
$, and its 
Fourier expansion is taken of the form
\begin{equation}
E_{\kappa,\,\chi}^{*}(z)={L^{\{p\}}(1-\kappa,\chi) \over 2}+\sum_{m=1}^{\infty}
\sigma_{\kappa-1,\,\chi}^{\{p\}}(m) \,q^m, 
\end{equation}
where $L^{\{p\}}(s,\chi):=(1-\chi(p)\,p^{-s})\,L(s,\chi)=
\prod_{l\ne p} (1-\chi(l)\,l^{-s})^{-1}$ and 
\[
\sigma_{\kappa-1,\,\chi}^{\{p\}}(m):=\sum_{\scriptstyle 0<d\, \mid\,m, \atop 
{\scriptstyle \gcd(d,p)=1}} \chi(d)\,d^{\kappa-1}. 
\]
It turns out that for $E_{\kappa,\,\chi}^{*}$, all the Hecke eigenvalues outside $p$ agree with $E_{\kappa,\,\chi}$,   
but the eigenvalue of Atkin's $U_p$-operator, which is converted from the $p$-th Hecke operator $T_p$ (cf. \cite{A-L70, Miy06}), is 
$\sigma_{\kappa-1,\,\chi}^{\{p\}}(p)=1$. Namely, we have 
\[
L(s,E_{\kappa,\,\chi}^{*})=
\zeta(s)L^{\{p\}}(s-\kappa+1,\chi). 
\]
This type of 
normalization $E_{\kappa,\,\chi} \mapsto E_{\kappa,\,\chi}^{*}$ given by 
eliminating the $p$-part of every Fourier coefficient or equivalently, eliminating 
the latter half of the Euler factor at $p$ was firstly introduced by Serre \cite{Ser73}, and 
we refer to it as the {\it ordinary\footnote{In general, a Hecke eigenform $f$ is said to be {\it ordinary} at $p$ if the eigenvalue of $U_p$ (or $T_p$) 
is a $p$-adic unit. } $p$-stabilization}. 
Here we should mention that every Fourier coefficient of $E_{\kappa,\,\chi}^*$ depends on the weight $\kappa$ $p$-adically. 
Indeed, it follows immediately from Fermat's little theorem that for each positive integer $m$, 
the function $\kappa \mapsto \sigma_{\kappa-1,\,\chi}^{\{p\}}(m)$ can be extended to an 
analytic function defined on the ring of $p$-adic integers $\mathbb{Z}_p$. 
In addition, the constant term $L^{\{p\}}(1-\kappa,\chi)/2$ is also interpolated by the $p$-adic $L$-function in the sense of Kubota-Leopoldt 
or Deligne-Ribet \cite{D-R80}.  
This fact turns out to be the following theorem due to Hida and Wiles: 

\begin{fact}[cf. Proposition 7.1.1 in \cite{Hid93}; Proposition 1.3.1 in \cite{Wil88}]

{\it Suppose that a Dirichlet character $\chi$ modulo $M$ is either trivial or primitive, and 
$M$ is not divisible by $p$. 
Let $\Lambda=\mathbb{Z}_p[\chi][[X]]$ be the one-variable power series ring over $\mathbb{Z}_p[\chi]$  
and $F_{\Lambda}$ the field of fractions of $\Lambda$, respectively. 
For each integer $a$ with $0 \le a < p-1$, 
there exists a formal Fourier expansion 
\[
\mathcal{E}_{\chi{\omega^a}}(X)
=\sum_{m=0}^{\infty} \mathcal{A}_{\chi\omega^a}(m;\,X)\, q^m \in 
F_{\Lambda}
[[q]],
\]
where $\omega
: \mathbb{Z}_p^{\times} \twoheadrightarrow \mu_{p-1}=\{ x \in \mathbb{Z}_p^{\times} \mid x^{p-1}=1\}
$\footnote{Since $\mu_{p-1} \simeq (\mathbb{Z}/p\mathbb{Z})^{\times}$, $\omega$ can be identified with a Dirichlet character modulo $p$ in a natural way. } denotes the Teichm\"uller character,
such that for each integer $\kappa>2$ with $\chi(-1)=(-1)^{\kappa}$ and $\kappa\equiv a \pmod{p-1}$\footnote{The latter condition is regarded as if $\omega^{a-\kappa}$ corresponds to the trivial Dirichlet character. }, we have 
\[
\mathcal{E}_{\chi{\omega^a}}
((1+p)^{\kappa} -1)=E_{\kappa,\,\chi}^{*}. 
\]
Moreover, if $\varepsilon : 1+p\mathbb{Z}_p \to (\mathbb{Q}_p^{\rm alg})^\times$ is a 
character of 
exact order $p^r$ for some integer $r \ge 0$\footnote{Namely, $\varepsilon$ can be also regarded as a primitive Dirichlet character of conductor $p^{r+1}$. }, then for each integer $\kappa\ge 2$, 
we have 
\[
\mathcal{E}_{\chi{\omega^a}}
(\varepsilon(1+p)(1+p)^{\kappa} -1)=E_{\kappa,\,\chi\omega^{a-\kappa}\varepsilon}
\]
as long as $\omega^{a-\kappa}\varepsilon$ is non-trivial. }
\end{fact}

Strictly speaking, it turns out that $\mathcal{E}_{\chi\omega^a}(X) \in \Lambda[[q]]$ 
unless $\chi\omega^a$ is trivial. 
However, we note that if $\chi\omega^a$ is trivial, then $\mathcal{E}_{\chi\omega^a}
(X)$ (more precisely, the constant term $\mathcal{A}_{\chi\omega^a}(0;\,X)$) has a simple pole 
at $X=0$ and thus   
$\overline{\mathcal{E}}_{\chi\omega^a}
(X):=X\cdot\mathcal{E}_{\chi\omega^a}
(X) \in \Lambda[[q]]$. 
Hence the above-mentioned fact implies 
that $\mathcal{E}_{\chi\omega^a}(X)$ or $\overline{\mathcal{E}}_{\chi\omega^a}(X)$ according as $\chi\omega^a$ is non-trivial or trivial, 
is a $\Lambda$-adic form of level $Mp^{\infty}$ associated with character $\chi\omega^a$, 
which interpolates 
families of non-cuspidal ordinary Hecke eigenforms
$\{E_{\kappa,\,\chi}^*\}$ and $\{E_{\kappa,\,\chi\omega^{a-\kappa}\varepsilon}\}$ 
or their constant multiples, 
given by varying the weight $\kappa$ 
$p$-adically analytically (cf. \cite{Wil88,Hid93}). 
In this context, we refer to it as the {\it ordinary $\Lambda$-adic Eisenstein series} of genus 1 and level $Mp^{\infty}$ associated with character $\chi\omega^a$. 

\vspace*{2mm}
The aim of the present article is to 
formulate a similar statement in the case of Siegel modular forms, that is, 
automorphic forms on the symplectic 
group ${\rm Sp}(2n)_{/\mathbb{Q}}$ of an arbitrary genus $n\ge1$. 
Let us explain how it goes briefly: 
Given a positive integer $\kappa>n+1$ and a Dirichlet character $\chi$ modulo a positive integer $M$, 
let $E_{\kappa,\,\chi}^{(n)}$ be the classical (holomorphic) Siegel Eisenstein series  
of weight $\kappa$ and nebentypus character $\chi$ for $\Gamma_0(M)^{(n)} \subset {\rm Sp}(2n,\mathbb{Z})$, 
to be described in the subsequent \S2, 
whenever either 
\begin{enumerate}
\item[{\rm (i)}] $M=1$, and thus, $\chi$ is trivial or \smallskip
\item[{\rm (ii)}] $M>1$ is odd, $\chi$ is primitive and 
$\chi^2$ is locally non-trivial at every prime $l \mid M$\footnote{Namely, 
if we factor $\chi$ as $\chi = \prod_{l\,\mid\, M} \chi_l$, then for each prime factor $l$ of $M$, $\chi_l$ is not a quadratic character. }. 
\end{enumerate} 
For an arbitrary prime number $p \nmid M$, we define, in \S4, a certain 
$p$-stabilization map $E_{\kappa,\,\chi}^{(n)} \mapsto (E_{\kappa,\,\chi}^{(n)})^*$ 
by means of the action of a linear combination of $(U_{p,n})^i$\,'s for $i=0,1, \cdots, n$, 
where $U_{p,n}$ denotes a generalized Atkin 
$U_p$-operator to be defined in \S 3 below, 
so that the associated eigenvalue of $U_{p,n}$ is $1$. 
This formulation is natural 
from the viewpoint of which 
in the case where $n=1$, $E_{\kappa,\,\chi}^{*}=(E_{\kappa,\,\chi}^{(1)})^*$ can be easily written 
in terms of Atkin's original operator $U_p=U_{p,1}$ as 
\begin{equation}
E_{\kappa,\,\chi}^{*} = E_{\kappa,\,\chi} \,\|_{\kappa}\, (U_p -\chi(p)\,p^{\kappa-1}). \tag{1'}
\end{equation}
Following in the tradition of Skinner-Urban \cite{S-U06}, in the case where $n >1$, 
we call $(E_{\kappa,\,\chi}^{(n)})^*$ {\it semi-ordinary} at $p$. 
Moreover, we derive an explicit formula for all Fourier coefficients of $(E_{\kappa,\,\chi}^{(n)})^*$, 
which can be regarded 
as a natural generalization of the equation (2) (cf. Theorem 4.2 below). 
As a consequence of the above issues, we first show in Theorem 4.4 below the existence of a formal Fourier expansion 
with coefficients in $F_{\Lambda}
$ associated to $\chi\omega^a$ with some nonnegative integer $a$, 
whose specialization at $X=(1+p)^\kappa -1$ coincides with the semi-ordinary $p$-stabilized Siegel Eisenstein series $(E_{\kappa,\,\chi}^{(n)})^*$ 
for any $\kappa$ taken as above.  
Finally in \S 5, for a fixed odd prime number $p$, we show that a suitable constant multiplication of  
the above-mentioned formal Fourier expansion 
gives rise to a $\Lambda$-adic form of genus $n$ and level $Mp^{\infty}$ (or tame level $M$ in the sense of Taylor \cite{Tay88}) 
associated with character $\chi\omega^a$, 
which can be viewed as a satisfactory generalization of Fact 1.1 (cf. Theorem 5.4 below). 

\smallskip

\begin{notation}
We summarize here some notation we will use in the sequel. 
We denote by $\mathbb{Z},\, \mathbb{Q},\, \mathbb{R},$ and $\mathbb{C}$ the ring of integers, fields of rational numbers, real numbers and complex numbers, 
respectively. Let $\mathbb{Q}^{\rm alg}$ denote the algebraic closure of $\mathbb{Q}$ sitting inside $\mathbb{C}$. 
We put $e(x)=\exp(2\pi \sqrt{-1}x)$ for $x \in \mathbb{C}$. 
Given a prime number $p$, we denote by $\mathbb{Q}_p$, $\mathbb{Z}_p$ and $\mathbb{Z}_p^{\times}$ 
the field of $p$-adic numbers, the ring of $p$-adic 
integers and the group of $p$-adic units, 
respectively. 
Hereinafter, given a prime number $p$, we fix an algebraic closure $\mathbb{Q}_p^{\rm alg}$ of $\mathbb{Q}_p$ 
and an embedding $\iota_p : \mathbb{Q}^{\rm alg}\hookrightarrow \mathbb{Q}_p^{\rm alg}$ 
once for all. Let ${\rm val}_p$ denote the $p$-adic valuation on $\mathbb{Q}_p^{\rm alg}$ 
normalized so that ${\rm val}_p(p)=1$, and $|*|_p$ the corresponding norm on $\mathbb{Q}_p^{\rm alg}$, respectively.  
Let $e_p$ be the continuous additive character of $\mathbb{Q}_p^{\rm alg}$ such that $e_p(x)=e(x)$ for all $x \in \mathbb{Q}$. 
Let $\mathbb{C}_p$ be the completion of the normed space $(\mathbb{Q}_p^{\rm alg}, |*|_p)$.  
We also fix, once for all, an isomorphism 
$\widehat{\iota}_p: \mathbb{C} \stackrel{\sim}{\to} \mathbb{C}_p$ 
such that the diagram \vspace*{-1mm}
\[\vspace*{-3mm}
\begin{array}{lcl}
\hspace*{1mm}\mathbb{C} & 
\underset{\simeq}{\stackrel{\widehat{\iota}_p}{\to}
}
& \hspace*{1mm}\mathbb{C}_p \\
\hspace*{2mm}\rotatebox{90}{$\hookrightarrow$} & 
& \hspace*{2mm}\rotatebox{90}{$\hookrightarrow$} \\
\mathbb{Q}^{\rm alg} & \underset{\iota_p}{\hookrightarrow} & \mathbb{Q}_p^{\rm alg}
\end{array} \vspace*{2mm}
\]
is commutative. Given a prime number $p$, put 
\[
{\bf p}=\left\{\begin{array}{cl}
4 & \text{ if $p=2$}, \\
p & \text{ otherwise}. 
\end{array}\right.
\]
Let $\mathbb{Z}_{p,{\rm tor}}^{\times}
$ be the torsion subgroup of $\mathbb{Z}_p^{\times}$, that is, $\mathbb{Z}_{2,{\rm tor}}^{\times}
=\{\pm 1\}$ and $\mathbb{Z}_{p,{\rm tor}}^{\times}=\mu_{p-1}$ 
if $p \ne 2$. We define the 
Teichm\"uller character 
$\omega : \mathbb{Z}_p^{\times} \twoheadrightarrow \mathbb{Z}_{p,{\rm tor}}^{\times}$ by putting $\omega(x)=\pm 1$ according as $x \equiv \pm 1 \pmod{4\mathbb{Z}_2}$ 
if $p=2$, and 
$\omega(x)=\displaystyle\lim_{i \to \infty} x^{p^i}$ for $x \in \mathbb{Z}_p^{\times}$ if $p \ne 2$. 
In addition, for each $x \in \mathbb{Z}_p^{\times}$, put $\langle x \rangle:=\omega(x)^{-1}x \in 1+{\bf p}\mathbb{Z}_p$. Then we have a canonical isomorphism 
\vspace*{-2mm} 
\[
\begin{array}{ccc}
\mathbb{Z}_p^{\times} &\stackrel{\sim}{\to}& \mathbb{Z}_{p,{\rm tor}}^{\times}
\times (1 + {\bf p} \mathbb{Z}_p) \\[1mm]
x &\mapsto& (\omega(x),\, \langle x \rangle). 
\end{array}
\vspace*{-1mm}
\]
We note that the maximal torsion-free subgroup 
$1+{\bf p}\mathbb{Z}_p$ of $\mathbb{Z}_p^{\times}$ is topologically cyclic, that is, 
$1+{\bf p}\mathbb{Z}_p=(1+{\bf p})^{\mathbb{Z}_p}$. 
As already mentioned in \S1, the Teichm\"uller character $\omega$ gives rise to a Dirichlet character $\omega:(\mathbb{Z}/{\bf p} \mathbb{Z})^{\times} \simeq \mathbb{Z}_{p,{\rm tor}}^{\times}  
\hookrightarrow \mathbb{C}_p^{\times} \simeq \mathbb{C}^{\times}$. 
Let $\varepsilon : 1+{\bf p}\mathbb{Z}_p \to 
(\mathbb{Q}_p^{\rm alg})^{\times}$ be a character of finite order. More precisely, if
$\varepsilon$ has exact order $p^m$ for some nonnegative integer $m$, then $\varepsilon$ optimally factors through 
\[
1+{\bf p}\mathbb{Z}_p/(1+{\bf p}\mathbb{Z}_p)^{p^m} 
\simeq (1+{\bf p})^{\mathbb{Z}_p}/(1+{\bf p})^{p^{m}\mathbb{Z}_p} 
\simeq \mathbb{Z}_p/p^{m}\mathbb{Z}_p 
\simeq \mathbb{Z}/p^m \mathbb{Z}.
\]  
Since $(\mathbb{Z}/p^{m}{\bf p} \mathbb{Z})^{\times} \simeq (\mathbb{Z}/{\bf p} \mathbb{Z})^{\times} \times \mathbb{Z}/p^m \mathbb{Z}$, we may naturally regard $\varepsilon$ as a Dirichlet character  
of conductor $p^{m}{\bf p}$. 
By abuse of notation, we will often identify such $p$-adic characters $\omega$ and $\varepsilon$ with 
the corresponding Dirichlet characters whose conductors are powers of $p$ in the sequel. 

\medskip

Given a positive integer 
$n$, let ${\rm GSp}(2n)
$ be the group of symplectic similitudes over $\mathbb{Q}$, that is, 
\[
{\rm GSp}(2n)
:=
\left\{\, g \in {\rm GL}(2n)\,
\right|  \left. {}^t g \,
J\,
g = \nu(g) 
J 
\textrm{ for some } \nu(g) \in \mathbb{G}_m
\,\right\}
, \]
where 
$
J =
\left[
\begin{smallmatrix}
0_n
 & -1_n \\
1_n & 
0_n
\end{smallmatrix}
\right] 
$ with the $n \times n$ unit (resp. zero) matrix $1_n$ {(resp. $0_n$)}, 
and 
${\rm Sp}(2n)$ the derived group of ${\rm GSp}(2n)$ characterized by the exact sequence 
\[
1 \to {\rm Sp}(2n) \to {\rm GSp}(2n) \stackrel{\nu}{\to} \mathbb{G}_m
 \to 1, 
\]
respectively. Namely, ${\rm GSp}(2)={\rm GL}(2)$ and ${\rm Sp}(2)={\rm SL}(2)$ in this setting. 
We note that every real point $g=\left[\begin{smallmatrix}
A & B \\
C & D
\end{smallmatrix}
\right] \in {\rm GSp}(2n,\mathbb{R})$ with $\nu(g)>0$, where $A,\,B,\,C,\,D \in {\rm Mat}_{n \times n}(\mathbb{R})$, acts on the Siegel upper-half space 
\[
\mathfrak{H}_n := \left\{ Z =X+\sqrt{-1}\,Y\in {\rm Mat}_{n\times n}(\mathbb{C}) \left|\, {}^t Z =Z,\, 
Y > 0\ (\textrm{positive-definite})\right.\right\}
\]
of genus $n$ via the linear transformation $Z \mapsto g \langle Z \rangle=(AZ+B)(CZ+D)^{-1}$. 
If $F$ is a function on $\mathfrak{H}_n$, then for each $\kappa \in \mathbb{Z}$, 
we define 
the slash action of $g$ on $F$ by 
\[
(F\, \|_{\kappa}\, g)(Z)
:= \nu(g)^{n\kappa-n(n+1)/2} 
\det(CZ+D)^{-\kappa} F(g \langle Z \rangle).
\] 
For each positive integer $N$, we shall consider the following congruence subgroups of level $N$
for the full-modular group ${\rm Sp}(2n,\mathbb{Z})$: 
\[
\Gamma_0(N)^{(n)} \, ({\rm resp. }\ \Gamma_1(N)^{(n)})
:=\left\{ 
\gamma \in {\rm Sp}(2n,\mathbb{Z}) \left|\,\,
\gamma \equiv 
\left[\begin{array}{cc}
* & * \\
0_{n} & *
\end{array}\right] \left({\rm resp.}\left[\begin{array}{cc}
* & * \\
0_{n} & 1_{n}
\end{array}\right]\right) \!\!\!\pmod{N} 
\right. \right\}.
\]
For each $\kappa \in \mathbb{Z}
$, 
let us denote by 
$\mathscr{M}_{\kappa}(
\Gamma_1(N)
)^{(n)}
$ 
the space of ({\it holomorphic}) 
{\it Siegel modular forms} of genus $n$, weight $\kappa$ 
and level $N$, that is, $\mathbb{C}$-valued holomorphic functions $F$ on $\mathfrak{H}_{n}$ satisfying the following two conditions: 
\begin{center}
\vspace*{-2mm}
\begin{minipage}[t]{0.975\textwidth}
\begin{enumerate}[\upshape (i)]

\item $F\, \|_{\kappa}\, \gamma = F$ \, for any 
$\gamma \in 
\Gamma_1(N)^{(n)}
$. \vspace*{1mm}

\item For each $\gamma \in {\rm Sp}(2n,\mathbb{Z})$, the function $F\, \|_{\kappa}\, \gamma$ possesses a Fourier expansion of the form
\[
(F\, \|_{\kappa}\, \gamma)(Z) = 
\displaystyle\sum_{T \in {\rm Sym}_n^{*}(\mathbb{Z})}
A_{F,\gamma}(T)\, e({\rm tr}(TZ)),
\]

\item[\hspace*{6mm}] 
where 
${\rm Sym}_n^{*}(\mathbb{Z})$ is 
the set of all half-integral symmetric matrices of degree $n$ over $\mathbb{Z}$, namely,     
\[
\hspace*{10mm}
{\rm Sym}_n^{*}(\mathbb{Z}):=
\{T=\left[\,t_{ij}\,\right] \in {\rm Sym}_{n}(\mathbb{Q})\,|\,
t_{ii},\,2t_{ij} \in \mathbb{Z}\,\,(1 \leq i < j \leq n)
\},\]
and ${\rm tr}(*)$ denotes the trace.  
Then it is satisfied that 
\begin{center}
$A_{F,\gamma}(T)=0$ unless $T \ge 0$ (positive-semidefinite) 
\end{center}
for all $\gamma \in {\rm Sp}(2n,\mathbb{Z})$. 
\end{enumerate}
\end{minipage}
\end{center}
A Siegel modular form $F \in \mathscr{M}_{\kappa}(\Gamma_1(N))^{(n)}$ is said to be {\it cuspidal} 
if it is satisfied that $A_{F,\gamma}(T) = 0$ unless $T > 0$ 
for all $\gamma\in {\rm Sp}(2n,\mathbb{Z})$. We denote by $\mathscr{S}_{\kappa}(\Gamma_1(N))^{(n)}$ the subspace of $\mathscr{M}_{\kappa}(\Gamma_1(N))^{(n)}$ consisting of all cuspidal forms. 
Given a Dirichlet character $\chi : (\mathbb{Z}/N\mathbb{Z})^{\times} \to \mathbb{C}^{\times}$, we denote by  
$\mathscr{M}_{\kappa}(\Gamma_0(N),\,\chi)^{(n)}$ (resp. $\mathscr{S}_{\kappa}(\Gamma_0(N),\,\chi)^{(n)}$) the subspace of $\mathscr{M}_{\kappa}(\Gamma_1(N))^{(n)}$ (resp. $\mathscr{S}_{\kappa}(\Gamma_1(N))^{(n)}$) consisting of all forms $F$ with nebentypus character $\chi$, that is, 
\[
F\, \|_{\kappa}\, \gamma = \chi(\det D) F \, \textrm{ for any } \gamma
=\left[
\begin{smallmatrix}
  A & B  \\
  C & D  \\
\end{smallmatrix}
\right] 
\in \Gamma_0(N)^{(n)}.   
\]
In particular, whenever $\chi$ is trivial, we naturally write 
$\mathscr{M}_{\kappa}(\Gamma_0(N))^{(n)}=\mathscr{M}_{\kappa}(\Gamma_0(N), \,{\rm triv})^{(n)}$ 
and ${\mathscr{S}_{\kappa}(\Gamma_0(N))^{(n)}=\mathscr{S}_{\kappa}(\Gamma_0(N),\,{\rm triv})^{(n)}}$, 
respectively. 

\medskip

For a given pair of $Z=[\,z_{ij}\,] \in \mathfrak{H}_n$ and $T=[\,t_{ij}\,] \in {\rm Sym}_n^{*}(\mathbb{Z})$, put
\[
{\bf q}^{T}:=e({\rm tr}(TZ))
=\prod_{i=1}^{n} q_{ii}^{t_{ii}} \prod_{i<j \leq n} q_{ij}^{2t_{ij}}, 
\]
where $q_{ij}=e(z_{ij})\, (1\leq i \leq j \leq n)$. 
Since $\left[\begin{smallmatrix}
1_n & S \\
0_n & 1_n
\end{smallmatrix}\right] \in \Gamma_1(N)^{(n)} \subset \Gamma_0(N)^{(n)}$ for each $S \in {\rm Sym}_n(\mathbb{Z})$, 
we easily see that if $F \in \mathscr{M}_{\kappa}(
\Gamma_1(N)
)^{(n)}$ (or $\mathscr{M}_{\kappa}(\Gamma_0(N), \chi)^{(n)}$), then 
$F$ possesses a Fourier 
expansion of the form
\[
F(Z)=\sum_{\scriptstyle T \in {\rm Sym}_n^{*}(\mathbb{Z}), 
\atop {\scriptstyle T \ge 0}
} A_{F}(T)\, 
{\bf q}^{T},    
\]
which is regarded as belonging to the ring $\mathbb{C}[\,q_{ij}^{\pm 1}\,|\,1 \leq i < j \leq n\,][[q_{11},\cdots,q_{nn}]]$. 
Given a ring $R$, 
we write 
$$R[[{\bf q}]]^{(n)}:=R[\,q_{ij}^{\pm 1}\,|\,1 \leq i < j \leq n\,][[q_{11},\cdots,q_{nn}]],$$ 
in a similar fashion to the notation of the 
ring of formal $q$-expansions 
$
R[[q]]$.  
In particular, if $F\in \mathscr{M}_{\kappa}(
\Gamma_1(N)
)^{(n)}$ is a Hecke eigenform (i.e., a simultaneous eigenfunction of all Hecke operators whose similitude is coprime to $N$), then it is well-known that 
the field $K_F$ obtained by adjoining all Fourier coefficients (or equivalently, all Hecke eigenvalues) of $F$ to $\mathbb{Q}$ is an algebraic 
number field.  
Thus, by virtue of the presence of $\iota_p$ and $\widehat{\iota}_p$, we may regard 
$F \in K_F[[{\bf q}]]^{(n)}$ as sitting inside $\mathbb{C}[[{\bf q}]]^{(n)}$ and $\mathbb{C}_p[[{\bf q}]]^{(n)}$ interchangeably. 
For further details on the basic theory of Siegel modular forms set out above, see \cite{A-Z95} or \cite{Fre83}. In particular, a comprehensive introduction to the theory of elliptic modular forms and Hecke operators can be found in \cite{Miy06}. 
\end{notation}

\section{Siegel Eisenstein series for symplectic groups}

In this section, we review some elementary facts on the Siegel Eisenstein series 
defined for ${\rm Sp}(2n)_{/\mathbb{Q}}$ of an arbitrary genus $n \ge 1$. 
In particular, we describe a explicit form of its Fourier expansion according to some previous works of Shimura (e.g., \cite{Shim82,
Shim94b}), which is the starting point for the subsequent arguments. 

\vspace*{3mm}
Let $N$ be a positive integer and $\chi : (\mathbb{Z}/N\mathbb{Z})^{\times} \to \mathbb{C}^{\times}
$ a Dirichlet character, respectively. As mentioned in \S 1, for simiplicity, we restrict ourselves to either of the following cases:
\begin{enumerate}
\item[{\rm (i)}] $N=1$, that is, $\chi$ is trivial; \vspace*{1mm}
\item[{\rm (ii)}] $N>1$ is odd, $\chi$ is primitive and 
$\chi^2$ is locally non-trivial at every prime $l \mid N$. 
\end{enumerate} 
Given a positive integer $n$, if $\kappa$ is an integer with 
$\kappa> n+1$ and $\chi(-1)=(-1)^{\kappa}$, 
then the ({\it holomorphic}) 
{\it Siegel Eisenstein series} 
of genus $n$, weight $\kappa$ and level $N$ with nebentypus character $\chi$ 
is defined as follows: 
For each $Z \in \mathfrak{H}_{n}$, put 
\begin{eqnarray*}
E_{\kappa,\,\chi}^{(n)}(Z)
&:=& 2^{-[(n+1)/2]
} L(1-\kappa,\, \chi) \prod_{i=1}^{[n/2]} L(1-2\kappa+2i,\, \chi^2) \\
&{}& \times \sum_{\gamma=\left[\begin{smallmatrix}
* & * \\
C & D
\end{smallmatrix}\right] \in (P_{2n}\, \cap\, \Gamma_0(N)^{(n)})
 \backslash \Gamma_0(N)^{(n)}
 } {\chi}^{-1}(\det D) \det(CZ+D)^{-\kappa},
\end{eqnarray*}
where $L(s,\, \psi)$ denotes Dirichlet's $L$-function associated with some character $\psi$, and 
$P_{2n}$ the Siegel parabolic subgroup of ${\rm Sp}(2n)$ consisting of all matrices $g = \left[\begin{smallmatrix}
* & * \\
0_n & *
\end{smallmatrix}\right]$, respectively. 
%

\vspace*{3mm}
Let $r$ be a positive integer. 
For each rational prime $l$, let ${\rm Sym}_{r}^{*}(\mathbb{Z}_l)$ denote the set of all half-integral symmetric matrices of degree $r$ over $\mathbb{Z}_l$. 
Given a nondegenerate $S \in {\rm Sym}_r^{*}(\mathbb{Z}_l)$, 
we define a formal power series $b_l(S;\,X)$ in $X$ by  
\[
b_l(S;\,X):=
\sum_{R \in {\rm Sym}_{r}(\mathbb{Q}_l)/{\rm Sym}_{r}(\mathbb{Z}_l)} e_l({\rm tr}(SR)) 
X^{{\rm val}_l(\mu_R)},
\]
where $\mu_R=[\mathbb{Z}_l^{r} + \mathbb{Z}_l^{r} R : \mathbb{Z}_l^{r}]$. 
Put $\mathfrak{D}_S:=
2^{2[r/2]} \det S
$. 
We note that 
if $r$ is even, 
then $(-1)^{r/2}\mathfrak{D}_S \equiv 0 \text{ or }1 \pmod{4}$, and thus, we may decompose it 
into the form 
$$(-1)^{r/2}\mathfrak{D}_S = \mathfrak{d}_S\,\mathfrak{f}_S^2,$$ 
where $\mathfrak{d}_S$ is the fundamental discriminant 
of the quadratic field extension $\mathbb{Q}_l(\{(-1)^{r/2}\mathfrak{D}_S\}^{1/2})/\mathbb{Q}_l
$ and $\mathfrak{f}_S=\{(-1)^{r/2}\mathfrak{D}_S/\mathfrak{d}_S\}^{1/2} \in \mathbb{Z}_l$. 
Let $\xi_l : \mathbb{Q}_l^{\times} \to \{\pm 1,\,0\}$ denote the character defined by 
\[
\xi_l(x)
=\left\{
\begin{array}{cl}
1 & \text{if\, } {\mathbb{Q}_l(x^{1/2})}={\mathbb{Q}_l}, \\[0.75mm]
-1 & \text{if\, } {\mathbb{Q}_l(x^{1/2})}/{\mathbb{Q}_l} \text{ is unramified}, \\[0.75mm]
0 & \text{if\, } {\mathbb{Q}_l(x^{1/2})}/{\mathbb{Q}_l} \text{ is ramified}. 
\end{array}
\right.
\]
As shown in \cite{
Kit84,
Fei86, Shim94b}, for each nondegenerate $S \in {\rm Sym}_{r}^{*}(\mathbb{Z}_l)$, 
there exists a polynomial $F_l(S;\,X) \in \mathbb{Z}[X]$ whose constant term is $1$ 
such that $b_l(S;\,X)$ is decomposed as follows: 
\begin{equation}
b_l(S;\,X)
=
F_l(S;\,X)\times
\left\{
\begin{array}{ll}
\displaystyle{(1-X)
\prod_{i=1}^{r/2} (1-l^{2i} X^2) \over 1- 
\xi_l((-1)^{r/2}\det S
)\,
l^{r/2} X} & \textrm{ if $r$ is even}, \\[5mm]
(1-X)\prod_{i=1}^{(r-1)/2} (1-l^{2i} X^2) & \textrm{ if $r$ is odd}
\end{array}\right.
\end{equation}
(cf. Proposition 3.6 in \cite{Shim94b}). 
%
We note that 
$F_l(S;\,X)$ satisfies the functional equation 
\begin{equation}
F_l(S;\, l^{-r-1} X^{-1})=F_l(S;\, X) \times
\left\{
\begin{array}{ll}
(l^{r+1} X^2)^{- {\rm val}_l(\mathfrak{f}_S)}  & \textrm{ if $r$ is even}, \\[3mm]
\eta_l(S)(l^{(r+1)/2} X)^{- {\rm val}_l(\mathfrak{D}_S)}  & \textrm{ if $r$ is odd},
\end{array}
\right.
\end{equation}
where 
\[
\eta_l(S):=h_l(S)\, (\det S,\,(-1)^{(r-1)/2}\det S)_l\, (-1,\,-1)_l^{(r^2-1)/8}
 \]
in terms of the Hasse invariant $h_l(S)$ in the sense of Kitaoka \cite{Kit84}
and the Hilbert symbol $(*,*)_l$ defined over $\mathbb{Q}_l$ (cf. Theorem 3.2 in \cite{Kat99}). 
Thus, it turns out that $F_l(S;\,X)$ has degree $2{\rm val}_l(\mathfrak{f}_S)$ or ${\rm val}_l(\mathfrak{D}_S)$ 
according as $r$ is even or odd. 
We easily see that $F_l(uS;\, X)=F_l(S;\,X)$ for each $u \in \mathbb{Z}_l^{\times}$, 
and that if $S$, $T \in {\rm Sym}_r^*(\mathbb{Z}_l)$ are equivalent over $\mathbb{Z}_l$, that is, 
$T={}^t U S U$ for some $U \in {\rm GL}(r, \mathbb{Z}_l)$, then $F_l(S;\,X)=F_l(T;\,X)$. 
For further details on the above-mentioned issues, see \cite{Kat99}. 

\begin{lem}
Let $n$, $\kappa$, $N$ and $\chi$ be taken as above. 
\begin{itemize}
\item[(I)] $E_{\kappa,\,\chi}^{(n)}
\in \mathscr{M}_{\kappa}(\Gamma_0(N),\,\chi)^{(n)}$ and it is a 
Hecke eigenform, that is, a simultaneous eigenfunction of Hecke operators defined at least for all primes not dividing the level $N$. 

\item[(I\hspace{-.1em}I)] 
Let us consider a Fourier expansion of $E_{\kappa,\,\chi}^{(n)}$ taken of the form 
\[
E_{\kappa,\,\chi}^{(n)}
(Z)=\sum_{\scriptstyle 
T \in {\rm Sym}_n^*(\mathbb{Z}), \atop {\scriptstyle 
T \ge 0}} A_{\kappa,\,\chi}(T)\,{\bf q}^T. 
\]
Then every coefficient $A_{\kappa,\,\chi}(T)$, which is invariant under $T \mapsto {}^t UTU$
for $U \in {\rm GL}(n,\mathbb{Z})$, is described as follows: \vspace*{1mm}
\begin{itemize}
\item[(I\hspace{-.1em}I\,a)] For $T=0_n \in {\rm Sym}_{n}^{*}(\mathbb{Z})$
, we have  
\[
A_{\kappa,\,\chi}(T)=2^{-[(n+1)/2]
} L(1-\kappa,\,\chi) \prod_{i=1}^{[n/2]} L(1-2\kappa+2i,\,\chi^2). 
\]
Therefore $E_{\kappa,\,\chi}^{(n)}$ is not cuspidal. \medskip

\item[(I\hspace{-.1em}I\,b)]If 
$T\in {\rm Sym}_n^{*}(\mathbb{Z})$ is taken of the form 
\[
T=\left[\begin{array}{c|c}
T' & {} \\ \hline
{} & 0_{n-r}
\end{array}\right]
\]
for some nondegenerate $T' \in {\rm Sym}_r^{*}(\mathbb{Z})$ with $0 < r \le n$ \text{\rm (i.e., ${\rm rank}\,T={\rm rank}\,T'=r$)}, then 
\begin{minipage}[t]{0.9\textwidth}
\begin{eqnarray*}
{A_{\kappa,\,\chi}(T)} 
&=& 2^{[(r+1)/2]-[(n+1)/2]} \prod_{i=[r/2]+1}^{[n/2]} L(1-2\kappa+2i,\chi^2) 
 \\ 
&&\times 
\left\{
\begin{array}{ll}
L(1-\kappa +r/2,\left({\mathfrak{d}_{T'} \over *}\right)\chi) 
\displaystyle\prod_{l\,\mid\, \mathfrak{f}_{T'} }  
F_l(T';\, \chi(l)\,l^{\kappa-r-1}) & \textrm{ if $r$ is even}, \\
 \displaystyle\prod_{l\,\mid\, \mathfrak{D}_{T'} }  
F_l(T';\, \chi(l)\,l^{\kappa-r-1}) & \textrm{ if $r$ is odd},  
\end{array}\right. 
\nonumber 
\end{eqnarray*}
\end{minipage}
\medskip

where 
$\left({\,\mathfrak{d}\, \over *}\right)$ denotes 
the Kronecker symbol.  
\end{itemize}
\end{itemize}

\end{lem}

\begin{rem}
For the convenience in the sequel, we make the convention for $r=0$, that 
$\mathfrak{D}_{S}=\mathfrak{d}_{S}=\mathfrak{f}_{S}=1$, and 
$F_l(S;\,X)=1$ for all primes $l$. 
This enables us to regard (I\hspace{-.1em}I\,a) as (I\hspace{-.1em}I\,b) for $r=0$. 
\end{rem}

\begin{proof}
The assertion (I) is well-known. 
The assertion (I\hspace{-.1em}I) can be obtained by exploiting an idea of 
Shimura \cite{Shim94b} as follows: 
Let $\mathbb{A}$ be the ring of adeles over $\mathbb{Q}$, $G_{2n}(\mathbb{A})={\rm Sp}(2n,\mathbb{A})$, 
$P_{2n}(\mathbb{A})=M_{2n}(\mathbb{A})N_{2n}(\mathbb{A})$ a Levi decomposition of the Siegel parabolic subgroup, where
\[
M_{2n} := \left\{ \left.
\left[\begin{array}{cc}
A & 0_n \\
0_n & {}^t\!A^{-1}
\end{array}\right] \ \right| A \in {\rm GL}(n) \right\}, \quad 
N_{2n}:=\left\{ \left.
\left[\begin{array}{cc}
1_n & B \\
0_n & 1_n
\end{array}\right] \ \right| {}^t B =B
\right\}, 
\]
and 
$\bfchi : {\mathbb{A}^{\times}/\mathbb{Q}^{\times}} \to \mathbb{C}^{\times}$ 
the unitary Hecke character corresponding to $\chi$, respectively. 
For each $s \in \mathbb{C}$, 
let 
${\rm Ind}_{P_{2n}(\mathbb{A})}^{G_{2n}(\mathbb{A})}(\bfchi \cdot |*|_{\mathbb{A}}^{s})$ denote the normalized smooth representation induced from the character 
of ${\rm GL}(n,\mathbb{A}) \simeq M_{2n}(\mathbb{A})$ defined by $A \mapsto \bfchi(\det A)\, |\det A|_{\mathbb{A}}^{s}$, 
where $|*|_{\mathbb{A}}$ denotes the 
norm on $\mathbb{A}$.  
Choosing a suitable section $\varphi^{(s)}
\in {\rm Ind}_{P_{2n}(\mathbb{A})}^{G_{2n}(\mathbb{A})}(\bfchi \cdot |*|_{\mathbb{A}}^{s})$, we define the Eisenstein series ${\bf E}
(\varphi^{(s)})
$ on $G_{2n}(\mathbb{A}
)$ by  
\[
{\bf E}
(\varphi^{(s)})(g):=\sum_{\gamma \in P_{2n}\backslash G_{2n}} 
\varphi^{(s)}(\gamma g),
\]
which converges absolutely for ${\rm Re}(s) \gg 0$, 
and 
${\bf E}
(\varphi^{(\kappa - (n+1)/2)})$
evaluates \[
2^{[(n+1)/2]}L(1-\kappa,\,\chi)^{-1}\prod_{i=1}^{[n/2]}L(1-2\kappa+2i,\,\chi^2)^{-1}
E_{\kappa,\,\chi}^{(n)}. \]
Thus, the desired equation can be obtained from an explicit formula for  
Fourier coefficients of ${\bf E}(\varphi^{(s)})$, more precisely, 
local Whittaker functions ${\rm Wh}_T(\varphi_v^{(s)})$ 
defined on $G_{2n}(\mathbb{Q}_v)$ 
for all places $v$ of $\mathbb{Q}$, 
where $\mathbb{A}=\prod_v \mathbb{Q}_v$, $\bfchi = \prod_v \bfchi_v$ and $\varphi^{(s)}=\prod_v \varphi_v^{(s)}$. 
Whenever $v$ is archimedean (i.e., $v=\infty$) or non-archimedean at which $\bfchi_v$ is unramified (i.e., $v$ is a prime $l$ not dividing $N$), 
it has been proved by Shimura (cf. Equations 4.34-35K in \cite{Shim82} and Proposition 7.2 in \cite{Shim94b}). 
Whenever $v$ is a non-archimedean place at which $\bfchi_v$ is ramified and $\bfchi_v^2$ is non-trivial, 
the local Whittaker function ${\rm Wh}_T(\varphi_v^{(s)})$ is described by Takemori \cite{Tak15} 
in which the key argument relies upon a functional equation of ${\rm Wh}_T(\varphi_v^{(s)})$ due to Ikeda \cite{Ike17} (generalizing \cite{Swe95}). 
\end{proof}

\section{generalized Atkin $U_p$-operator}

\vspace*{2mm}

In this section, we recall the theory of Atkin's $U_p$-operator and its generalization. For further details on the facts set out below, 
see, for instance, \cite{A-Z95,Boe05,Tay88}. 

\medskip

Let $p$ be a prime number and $N$ a positive integer, respectively. 
If $p \mid N$, the 
coset decomposition
\[
\Gamma_0(N)^{(1)} \left[\begin{array}{cc}
1 & 0 \\
0 & p 
\end{array}\right]\Gamma_0(N)^{(1)}
=\bigsqcup_{s=0}^{p-1} \Gamma_0(N)^{(1)}\left[\begin{array}{cc}
1 & s \\
0 & p 
\end{array}\right]
\]
induces the following linear operator $U_p=U_{p,1}$ 
on $\mathscr{M}_{\kappa}(\Gamma_0(N),\,\chi)^{(1)}$: 
For each $f \in \mathscr{M}_{\kappa}(\Gamma_0(N),\,\chi)^{(1)}
$, put
\[
(f\,\|_{\kappa} \, U_p)(z) := \sum_{s=0}^{p-1} \left(f\,\|_{\kappa}\left[\begin{array}{cc}
1 & s \\
0 & p 
\end{array}\right]\right)\!(z)
= p^{-1} \sum_{s=0}^{p-1} f\!\left({z+s \over p}\right) \in \mathscr{M}_{\kappa}(\Gamma_0(N),\,\chi)^{(1)}. 
\]
We easily see that it is written in terms of Fourier expansions as 
\begin{equation}
f(z)=\sum_{m=0}^{\infty} a(m) \,q^m \longmapsto (f\,\|_{\kappa} \, U_p)(z)=\sum_{m=0}^{\infty} a(pm) \,q^m 
\end{equation}
and this is still valid even if $p \nmid N$, 
however it maps from $\mathscr{M}_{\kappa}(\Gamma_0(N),\,\chi)^{(1)}$ to $\mathscr{M}_{\kappa}(\Gamma_0(Np),\,\chi)^{(1)}$ in this case. 
We refer to the operator $U_{p}$, regardless of whether $p \mid N$ or not, as Atkin's $U_p$-operator. 
\begin{rem}
Obviously, $U_p$ coincides with the usual Hecke operator $T_p$ if $p \mid N$. However, $U_p$ is slightly different from $T_p$ in general:    
\[
\Gamma_0(N)^{(1)} \left[\begin{array}{cc}
1 & 0 \\
0 & p 
\end{array}\right]\Gamma_0(N)^{(1)}
=\bigsqcup_{s=0}^{p-1} \Gamma_0(N)^{(1)}\left[\begin{array}{cc}
1 & s \\
0 & p 
\end{array}\right] \sqcup \Gamma_0(N)^{(1)}\left[\begin{array}{cc}
p & 0 \\
0 & 1 
\end{array}\right]
\]
if $p \nmid N$. 
\end{rem}
Similarly, if $n>1$ and $p \mid N$, the following $n$ double-coset operators at $p$ are relevant for Siegel modular forms of genus $n$ and level $N$: 
\[
U_{p,i}:=
\left\{\begin{array}{ll}
{\Gamma_0(N)^{(n)}\, 
{\rm diag}(\underbrace{1,\cdots,1}_{i},\underbrace{p,\cdots,p}_{n-i},
\underbrace{p^2,\cdots,p^2}_{i},\underbrace{p,\cdots,p}_{n-i})
\, \Gamma_0(N)^{(n)}} & \textrm{ if }1 \le i \le n-1, \\[8mm] 
{\Gamma_0(N)^{(n)}\, 
{\rm diag}(\underbrace{1,\cdots,1}_{n},\underbrace{p,\cdots,p}_{n})
\, \Gamma_0(N)^{(n)}} & \hspace*{-23mm}\textrm{ if }i=n.  
\end{array}\right. 
\]
We note that if $N$ is divisible by $p$, 
these operators $U_{p,1},\,\cdots,\, U_{p,n-1}$ and $U_{p,n}$ generate the dilating Hecke algebra 
at $p$ acting on $\mathscr{M}_{\kappa}(\Gamma_0(N),\,\chi)^{(n)}$. 
In particular, we are interested in the operator $U_{p,n}$ which plays a central role among them. 
Namely, we define the operator $U_{p,n}$ on $\mathbb{C}_p[[{\bf q}]]^{(n)}$ by 
\begin{equation}
F=
\sum_{T \ge 0} A_{F}(T) \,{\bf q}^T \longmapsto 
F\,\|_{\kappa}\, U_{p,n}=\sum_{T \ge 0} A_{F}(pT) \,{\bf q}^T. 
\end{equation}
Indeed, 
we easily see that if $F \in \mathscr{M}_{\kappa}(\Gamma_0(N),\,\chi)^{(n)}$ 
with some positive integers $\kappa$, $N$ and a Dirichlet character $\chi$, 
then 
\[
F \,\|_{\kappa}\, U_{p,n} \in \left\{
\begin{array}{ll}
\mathscr{M}_{\kappa}(\Gamma_0(N),\,\chi)^{(n)} & \textrm{ if }p \mid N, \\[2mm]
\mathscr{M}_{\kappa}(\Gamma_0(Np),\,\chi)^{(n)} & \textrm{ if }p \nmid N
\end{array}
\right. \vspace*{2mm}
\]
and the action of $U_{p,n}$ commutes with those of the Hecke operators defined outside $Np$. 

\section{Semi-ordinary $p$-stabilization of Siegel Eisenstein series}


Let us fix an odd integer $M$ and 
a prime number $p \nmid M$ (including $p=2$) once for all. 
In this section, for a pair of positive integers $n$, $\kappa$ and a Dirichlet character $\chi$ modulo $M$ taken as in \S 2, that is,  
\[
\kappa > n+1, \quad \chi(-1) = (-1)^{\kappa} \text{ \ and \  $\chi^2$ is locally non-trivial at every prime $l \mid M$ if $M>1$},   
\]
we introduce a certain $p$-stabilization of the Siegel Eisenstein series $E_{\kappa,\,\chi}^{(n)}\in \mathscr{M}_{\kappa}(\Gamma_0(M),\,\chi)$
which can be viewed as a natural generalization of the ordinary $p$-stabilization 
$E_{\kappa,\,\chi}^{(1)} \mapsto (E_{\kappa,\,\chi}^{(1)})^*$ (cf. Equations (1) and (1')). Moreover, 
we derive an explicit form of the associated Fourier expansion 
\[
(E_{\kappa,\,\chi}^{(n)})^{*}(Z)=\sum_{\scriptstyle 
T \in {\rm Sym}_n^*(\mathbb{Z}), \atop {\scriptstyle 
T \ge 0}} A_{\kappa,\,\chi}^{*}(T)\,{\bf q}^T, 
\]
which is expressed in a similar fashion to (2), and conclude its $p$-adic interpolation problem. \\

To begin with, we introduce 
the following two polynomials in $X$ and $Y$: 
\begin{eqnarray}
\mathcal{P}_{p}^{(n)}(X,\,Y) 
&:=&(1-p^{n}XY)\prod_{i=1}^{[n/2]}(1-p^{2n-2i+1}X^2Y), \nonumber \\
\mathcal{R}_{p}^{(n)}(X,\,Y) 
&:=&\prod_{j=1}^{n} (1- p^{j(2n-j+1)/2}X^j Y). \nonumber 
\end{eqnarray}
In addition, let us denote by $\widetilde{\mathcal{R}}_{p}^{(n)}(X,\,Y)$
the reflected polynomial of $\mathcal{R}_{p}^{(n)}(X,\,Y)$ with respect to $Y$, that is,  
\begin{equation}
\widetilde{\mathcal{R}}_{p}^{(n)}(X,\,Y) := Y^{n}\, \mathcal{R}_{p}^{(n)}(X,\,Y^{-1}) = \prod_{j=1}^{n} (Y- p^{j(2n-j+1)/2}X^j). 
\end{equation}

\begin{rem}
Whenever $n=1$ and $2$, a straightforward calculation yields 
\[\left\{
\begin{array}{l}
\mathcal{P}_p^{(1)}(X,\,Y)=\mathcal{R}_p^{(1)}(X,\,Y)=1-pXY, \\[2mm]
\mathcal{P}_p^{(2)}(X,\,Y)=\mathcal{R}_p^{(2)}(X,\,Y)=(1-p^2XY)(1-p^3X^2Y). 
\end{array}\right.\]
We note that if $n >2$, then $\mathcal{P}_{p}^{(n)}(X,\,Y)\ne\mathcal{R}_{p}^{(n)}(X,\,Y)$,
however, $\mathcal{P}_{p}^{(n)}(X,\,1)$ divides $\mathcal{R}_{p}^{(n)}(X,\,1)$ in general. 
For readers' convenience, we reveal the origins of these two polynomials here:  
Obviously, the former 
$\mathcal{P}_{p}^{(n)}(X,\,Y)$ is relevant to 
the local factor of the Fourier coefficient 
$A_{\kappa,\,\chi}(0_n)$ described in Lemma 2.1\,(I\hspace{-.1em}I\,a) at $p$:
\begin{equation}
\mathcal{P}_{p}^{(n)}(\chi(p)\,p^{\kappa-n-1},\,1)=(1-\chi(p) \,p^{\kappa-1})\prod_{i=1}^{[n/2]}(1-\chi^2(p) \,p^{
2\kappa-2i-1}). 
\end{equation}
The latter $\mathcal{R}_{p}^{(n)}(X,\,Y)$ has been introduced by Kitaoka \cite{Kit86} and B\"ocherer-Sato \cite{B-S87} to describe the denominator of the formal power series 
$
\sum_{m=0}^{\infty} F_p(p^mS;\,X)Y^m 
$ for each nondegenerate $S \in {\rm Sym}_n^{*}(\mathbb{Z}_p)$ (cf. Equation (11) below). 
\end{rem}

\vspace*{2mm}
Now, we introduce a $p$-stabilization of the Siegel Eisenstein series $E_{\kappa,\,\chi}^{(n)}\in \mathscr{M}_{\kappa}(\Gamma_0(M),\,\chi)^{(n)}$ 
in terms of a linear combination of 
$(U_{p,n})^i = \underbrace{U_{p,n} \circ \cdots \circ U_{p,n}}_{i}$ for $i=0,1,\cdots, n$ 
as follows: 

\begin{thm}
For a pair of $n$, $\kappa$ and $\chi$ taken as above, 
put 
\begin{equation}
(E_{\kappa,\,\chi}^{(n)})^*:=
{\mathcal{P}_{p}^{(n)}(\chi(p)\,p^{\kappa-n-1},\,1) \over \mathcal{R}_{p}^{(n)}(\chi(p)\,p^{\kappa-n-1},\,1)}
\cdot E_{\kappa,\,\chi}^{(n)}\,\|_{\kappa}\,\widetilde{\mathcal{R}}_{p}^{(n)}(\chi(p)\,p^{\kappa-n-1},\,U_{p,n}). 
\end{equation}
Then we have 
\smallskip

\begin{enumerate}
\item[{\rm (I)}] $(E_{\kappa,\,\chi}^{(n)})^* \in \mathscr{M}_{\kappa}(\Gamma_0(Mp),\,\chi)^{(n)}$ and it is a 
Hecke eigenform such that 
all the eigenvalues outside $Mp$ agree with those of $E_{\kappa,\,\chi}^{(n)}\in \mathscr{M}_{\kappa}(\Gamma_0(M),\,\chi)^{(n)}$. 

\medskip

\item[{\rm (I\hspace{-.1em}I)}]If $0 \le T \in {\rm Sym}_{n}^{*}(\mathbb{Z})$ is 
taken of the form 
\[
T=
\left[\begin{array}{c|c}
T' & {} \\
\hline
{} & 0_{n-r}
\end{array}
\right]
\] 
for some nondegenerate $T' \in {\rm Sym}_r^{*}(\mathbb{Z})$ with $0 \le r \le n$,  
then the $T$-th Fourier coefficient of $(E_{\kappa,\,\chi}^{(n)})^*$ 
is taken of the following form: \vspace*{-2mm}
\begin{center}
\begin{minipage}[t]{0.9\textwidth}  
\begin{eqnarray*}
A_{\kappa,\,\chi}^{*}(T)
&=& 2^{[(r+1)/2]-[(n+1)/2]} \prod_{i=[r/2]+1}^{[n/2]} L^{\{p\}}(1-2\kappa+2i,\chi^2) \nonumber \\ 
&& \times
\left\{
\begin{array}{ll}
L^{\{p\}}(1-\kappa+r/2, \, \left({\mathfrak{d}_{T'} \over *}\right)\chi) 
\displaystyle\prod_{\scriptstyle \,\,l\, \mid\, \mathfrak{f}_{T'}, \atop {\scriptstyle l \ne p}} 
F_l(T';\, 
\chi(l)\,l^{\kappa-r-1}) & \textrm{ if $r$ is even}, \\[5mm] 
\displaystyle\prod_{\scriptstyle \,\,l\, \mid\, \mathfrak{D}_{T'}, \atop {\scriptstyle l \ne p}} 
F_l(T';\, 
\chi(l)\,l^{\kappa-r-1}) & \textrm{ if $r$ is odd}. 
\end{array}\right. \nonumber
\end{eqnarray*}
\end{minipage}
\end{center}
\medskip

\noindent
Therefore we have  
$
(E_{\kappa,\,\chi}^{(n)})^*\,\|_{\kappa}\,U_{p,n} = (E_{\kappa,\,\chi}^{(n)})^*. 
$
\end{enumerate}
\end{thm}

The preceding theorem totally 
insists that the assignment $E_{\kappa,\,\chi}^{(n)} \mapsto (E_{\kappa,\,\chi}^{(n)})^{*}$ can be regarded as a $p$-stabilization 
which generalizes the ordinary $p$-stabilization of $E_{\kappa,\,\chi}=E_{\kappa,\,\chi}^{(1)}$ explained in \S 1. 
Whenever $n=2$, Skinner-Urban \cite{S-U06} has already dealt with a similar type of $p$-stabilization 
for some Siegel modular forms of genus $2$ so that the associated eigenvalue of $U_{p,2}$ is a $p$-adic unit as well. 
Accordingly, in the same context, we may 
call $(E_{\kappa,\,\chi}^{(n)})^{*}$ 
{\it semi-ordinary} at $p$ if $n \ge 2$. 
\begin{rem}
It should be mentioned that if $n >1$, $(E_{\kappa,\,\chi}^{(n)})^*$ may {\it not}\, satisfy the ordinary condition at $p$ in the sense of Hida (cf. \cite{Hid02,Hid04}). 
This disparity is inevitable at least for Siegel Eisenstein series of higher genus in general. 
For instance, in the case where $M=1$, because of the shape of associated Satake parameters at $p$, 
there is no way to produce from $E_{\kappa}^{(n)}=E_{\kappa,\,{\rm triv}}^{(n)} \in \mathscr{M}_{\kappa}(\Gamma_0(1))^{(n)}$ to 
a Hecke eigenform of level $p$ such that the associated eigenvalues of $U_{p,1},\cdots,\,U_{p,n-1}$ and $U_{p,n}$ are $p$-adic units simultaneously. 
However, 
as mentioned in \cite{S-U06}, it turns out that the semi-ordinary condition concerning only on the eigenvalue of $U_{p,n}$ 
is sufficient to adapt Hida's ordinary theory with some modification. (See also \cite{Pi11, B-P-S16}.)
\end{rem}

\vspace*{2mm}
\begin{proof}[Proof of Theorem 4.2]
The assertion (I) is obvious from Lemma 2.1\,(I) and the properties of $U_{p,n}$ explained in \S 3. 
Since $U_{p,n}$ does not effect on the Fourier coefficient $A_{\kappa,\,\chi}(0_n)$,  
the assertion (I\hspace{-.1em}I) for $r=0$ follows immediately from Lemma 2.1\,(I\hspace{-.1em}I\,a), Equations (8) and (9). 
Hereinafter, we suppose that $r >0$. It follows by the definition of $\widetilde{\mathcal{R}}_{p}^{(n)}(X,Y)$ (cf. Equation (7)) that
\begin{eqnarray*}
\widetilde{\mathcal{R}}_{p}^{(n)}(X,Y)
&=& 
\sum_{m=0}^{n} (-1)^m 
s_m\!\left(\{ p^{j(2n-j+1)/2} X^j \,|\, 1 \le j \le n\}\right) Y^{n-m}, 
\end{eqnarray*} 
where 
$s_m(\{X_1,\cdots,X_n\})$ denotes  
the $m$-th elementary symmetric polynomial in $X_1,\,\cdots,\,X_n$. 
Thus, by Lemma 2.1\,(I\hspace{-.1em}I\,b) and Equation (9), we have 
\begin{eqnarray*}
\lefteqn{
A_{\kappa,\,\chi}^{*}(T)
={\mathcal{P}_{p}^{(n)}(\chi(p)\,p^{\kappa-n-1},\,1) \over 
\mathcal{R}_{p}^{(n)}(\chi(p)\,p^{\kappa-n-1},\,1)} \cdot \,
2^{[(r+1)/2]-[(n+1)/2]}
\prod_{i=[r/2]+1}^{[n/2]} L(1-2\kappa+2i,\,\chi^2) }
 \\
&\times& \sum_{m=0}^{n} (-1)^m 
s_m\!\left(\{\chi^j(p)\,p^{j(2n-j+1)/2+j(\kappa-n-1)} \,|\, 1 \le j \le n\}\right) 
F_p(p^{n-m}T';\, \chi(p)\,p^{\kappa-r-1}) \\
&& \times 
\left\{
\begin{array}{ll}
L(1-\kappa+r/2, \, \left({\mathfrak{d}_{T'} \over *}\right)\chi) 
\displaystyle\prod_{\scriptstyle \,l\, \mid \, \mathfrak{f}_{T'}, \atop {\scriptstyle l \ne p}} 
F_l(T';\,\chi(l)\,l^{\kappa-r-1}) & \textrm{ if $r$ is even}, \\[5mm]
\displaystyle\prod_{\scriptstyle \,l\, \mid \, \mathfrak{D}_{T'}, \atop {\scriptstyle l \ne p}} 
F_l(T';\,\chi(l)\,l^{\kappa-r-1}) & \textrm{ if $r$ is odd}. 
\end{array}
\right. 
\end{eqnarray*}
Here we note that 
\[\left\{
\begin{array}{l}
\mathcal{R}_{p}^{(n)}(X,\,Y)= \mathcal{R}_{p}^{(r)}(p^{n-r}X,\,Y) 
\displaystyle\prod_{j=r+1}^{n} (1- p^{j(2n-j+1)/2}X^j Y), \\[5mm]
\mathcal{P}_{p}^{(n)}(X,\,1)= \mathcal{P}_{p}^{(r)}(p^{n-r}X,\,1) 
\displaystyle\prod_{i=[r/2]+1}^{[n/2]}(1-p^{2n-2i+1}X^2).
\end{array}\right.\]
Thus, to prove the assertion (I\hspace{-.1em}I), it suffices to show that 
the following equation holds valid for each 
nondegenerate $T' \in {\rm Sym}_{r}^{*}(\mathbb{Z}_p)$ with
$0 < r \le n$: 
\begin{eqnarray}
\lefteqn{\sum_{m=0}^{r} (-1)^m  
s_m\!\left(
\{p^{j(2r-j+1)/2} X^j \,|\, 1 \le j \le r\}
\right) F_p(p^{r-m}T';\,X)} \\
&=& 
\displaystyle{\mathcal{R}_{p}^{(r)}(X,\,1) \over \mathcal{P}_{p}^{(r)}(X,\,1)} \cdot 
\left\{
\begin{array}{ll}
\left(1 - 
\xi_p((-1)^{r/2}\det T'
)\, p^{r/2} X\right) 
 & \textrm{ if $r$ is even}, \\[5mm]
1 & \textrm{ if $r$ is odd}. 
\end{array}
\right. 
 \nonumber
\end{eqnarray}
(Indeed, the preceding equation (10) yields 
%
%
\begin{eqnarray*}
\lefteqn{
{\mathcal{P}_{p}^{(n)}(X,\,1) \over 
\mathcal{R}_{p}^{(n)}(X,\,1)}
 \sum_{m=0}^{n} (-1)^m 
s_m\!\left(\{p^{j(2n-j+1)/2}X^j \,|\, 1 \le j \le n\}\right) 
F_p(p^{n-m}T';\, p^{n-r}X)
} \\
&=& 
%
%
\displaystyle\prod_{i=[r/2]+1}^{[n/2]}(1-p^{2n-2i+1}X^2) \times 
\left\{
\begin{array}{ll}
\left(1 - 
\xi_p((-1)^{r/2}\det T'
)\, p^{n-r/2} X\right) & \textrm{ if $r$ is even}, \\[5mm]
1 & \textrm{ if $r$ is odd}, 
\end{array}
\right. 
\end{eqnarray*}
and hence, by evaluating this at $X=\chi(p)\,p^{\kappa-n-1}$, we obtain the desired equation. )
On the other hand,  
Theorem 1 in \cite{Kit86} (resp. Theorem 6 in \cite{B-S87}) states that 
for each nondegenerate $T' \in {\rm Sym}_{r}^{*}(\mathbb{Z}_p)$, 
the equation 
\begin{equation}
\sum_{m=0}^{\infty} F_p(p^m T';\,X)Y^m 
={\mathcal{S}_p
(T';\,X,\,Y)\over (1-Y)\mathcal{R}_{p}^{(r)}(X,\,Y)}
\end{equation}
holds for some polynomial $\mathcal{S}_p(T';\,X,\,Y) \in \mathbb{Z}[X,\,Y]$ if $p\ne 2$ (resp. $p=2$).  
Since the preceding equation yields 
\[
\sum_{m=0}^{r} (-1)^m  
s_m\!\left(
\{p^{j(2r-j+1)/2} X^j \,|\, 1 \le j \le r\}
\right) F_p(p^{r-m}T';\,X)=\mathcal{S}_p(T';\,X,\,1), 
\]
we may interpret Equation (10) as  
\begin{equation}
\mathcal{S}_p
(T';\,X,\,1)=
\displaystyle{\mathcal{R}_p^{(r)}(X,\,1) \over \mathcal{P}_p^{(r)}(X,\,1)} \times 
\left\{
\begin{array}{ll}
\left(1 - 
\xi_p((-1)^{r/2}\det T'
)\, p^{r/2} X\right) 
 & \textrm{ if $r$ is even}, \\[5mm]
1 & \textrm{ if $r$ is odd}. 
\end{array}
\right. 
\end{equation}
Whenever $r=1$, we easily see that  
$F_p(t;\,X)=\sum_{i=0}^{{\rm val}_p(t)} (pX)^i$ for each $t \in \mathbb{Z}_p \smallsetminus \{0\}$. Thus, we have 
\[
F_p(pt;\,X)-pX F_p(t;\,X)=1={\mathcal{R}_p^{(1)}(X,1) \over \mathcal{P}_p^{(1)}(X,1)} \quad {\rm (cf. \ Remark\ 4.1)}. 
\]
Whenever $r >1$, 
for a given $T$, let $\mathfrak{i}(T)$ denote the least integer $m$ such that $p^m T^{-1} \in {\rm Sym}_{r}^{*}(\mathbb{Z}_p)$. 
It is known that if $r=2$, then for each nondegenerate $T \in {\rm Sym}_2^{*}(\mathbb{Z}_p)$, 
the polynomial $F_p(T;\,X)$ admits the explicit form 
\begin{eqnarray*}
F_p(T;\,X) = \sum_{i=0}^{\mathfrak{i}(T)} (p^2 X)^i \left\{ 
\sum_{j=0}^{{\rm val}_p(\mathfrak{f}_T) - i} (p^3 X^2)^j
-\xi_p(
-\det T)\,pX 
\sum_{j=0}^{{\rm val}_p(\mathfrak{f}_T) - i-1} (p^3 X^2)^j \right\} 
\end{eqnarray*}
(cf. \cite{
Kat99}). Thus 
a simple calculation yields that 
\begin{eqnarray*}
\lefteqn{
F_p(p^2 T;\, X)
-(p^2 X+p^3 X^2) F_p(p T;\,X) 
+p^5 X^3 F_p(T;\,X)
} \\[1mm]
&=& 1 - \xi_p(-\det T)\, p X
= {\mathcal{R}_p^{(2)}(X,1)
\over 
\mathcal{P}_p^{(2)}(X,1)
} \times
\left(1 - \xi_p(-\det T)\, p X\right),
\end{eqnarray*}
and hence, Equation (12) also holds for $r=2$. 
Now, we suppose that $r>2$. 
We note that every nondegenerate $T \in {\rm Sym}_{r}^{*}(\mathbb{Z}_p)$ is equivalent, over $\mathbb{Z}_p$, to a canonical form
\begin{center}
$
T=\left[
\begin{array}{c|c}
T_1 & {}  \\
\hline
{} & T_2
\end{array}
\right]$ 
\end{center}
for some $
T_1 \in {\rm Sym}_{2}^{*}(\mathbb{Z}_p)$ and $T_2 \in {\rm Sym}_{r-2}^{*}(\mathbb{Z}_p)\cap{\rm GL}(r-2,\mathbb{Q}_p)$. 
It follows from Theorems 4.1 and 4.2 in \cite{Kat99} that 
\begin{eqnarray*}
{\mathcal{S}_p(T;\,X,\,1) \over 1-\xi_p(
(-1)^{r/2}\det T
)\,p^{r/2} X}&=&
{\mathcal{S}_p
(T_2;\,p^2X,\,1) \over 1- \xi_p(
(-1)^{r/2-1}\det T_2
)\,p^{r/2+1}X} \\
&& \times {(1-p^{(r-1)(r+2)/2}X^{r-1})(1 - p^{r(r+1)/2}X^r) 
\over 1- p^{r+1}X^2}
\end{eqnarray*}
if $r$ is even, and 
\[
\mathcal{S}_p(T;\,X,\,1)=\mathcal{S}_p(T_2;\,p^2X,\,1)\cdot
{
(1-p^{(r-1)(r+2)/2}X^{r-1})(1 - p^{r(r+1)/2}X^r)
\over 1- p^{r+2
}X^2
}
\]
if $r$ is odd.  
(See also \cite[\S 3]{Kat01}. ) 
Thus, it is proved by induction on $r$ that Equation (12) (and thus, (10)\,) holds in general. 
This completes the proof. 
\end{proof}

\vspace*{1mm}
As a straightforward conclusion of Theorem 4.2 above, we have  
\vspace*{1mm}
\begin{thm} 
Let $\chi$ be a Dirichlet character modulo 
$M$, where $M$ is odd and coprime to $p$, taken as above 
and $\omega^a$ a power of the Teichm\"uller character with $0 \le a <\varphi({\bf p})$\footnote{Here $\varphi$ denotes Euler's totient function.}, respectively. 
For each $n \ge 1$,  
there exists a formal Fourier expansion  
\[\mathcal{E}_{\chi\omega^a}^{(n)}
(X)=\sum_{\scriptstyle T \in {\rm Sym}_n^*(\mathbb{Z}), \atop {\scriptstyle T \ge 0}} 
\mathcal{A}_{\chi\omega^a}(T;\,X)
\, {\bf q}^T 
\in F_{\Lambda}[[{\bf q}]]^{(n)}, 
\]
where $F_{\Lambda}$ is the field of fractions of $\Lambda=\mathbb{Z}_p[\chi][[X]]$, 
such that for each positive integer 
$\kappa > n+1$ with $\chi(-1)=(-1)^{\kappa}$ and $\kappa \equiv a \pmod{\varphi({\bf p})}$, we have 
\[
\mathcal{E}_{\chi\omega^a}^{(n)}
((1+{\bf p})^{\kappa} -1)=
(E_{\kappa,\,\chi}^{(n)})^*
\in \mathscr{M}_{\kappa}(\Gamma_0(Mp),\,\chi)^{(n)}. 
\]   
Moreover, put 
\begin{eqnarray*}
\mathcal{B}^{(n)}(X)&:=&
\prod_{i=1}^{[n/2]} \{(1+{\bf p})^{-2i} (1+X)^2 -1\}
\prod_{j=0}^{[n/2]} \{(1+{\bf p})^{-j} (1+X) -1\} \\
&=& 
X \prod_{i=1}^{[n/2]} \{(1+{\bf p})^{-i} (1+X) -1\}^2
\cdot \{(1+{\bf p})^{-i} (1+X) +1\}.
\end{eqnarray*}
Then \,$\overline{\mathcal{E}}_{\chi\omega^a}^{\,(n)}
(X):=\mathcal{B}^{(n)}(X) \cdot 
\mathcal{E}_{\chi\omega^a}^{(n)}
(X)$ belongs to $\Lambda[[{\bf q}]]^{(n)}$. 
\end{thm}

\begin{proof}
As is well-known by Deligne-Ribet \cite{D-R80} (generalizing the previous work of Kubota-Leopoldt), 
associated to a quadratic Dirichlet character $\xi$, a Dirichlet character $\chi$ taken as above, 
and a power of the Teichm\"uller character $\omega^a$ with $0 \le a<\varphi({\bf p})$
, there exists $\overline{\mathcal{L}}(\xi\chi\omega^a;\, X) \in 
\Lambda
$ such that for each positive integer $k>1$, we have
\[
\overline{\mathcal{L}}(\xi\chi\omega^a;\,
(1+{\bf p})^k -1)=
\left\{
\begin{array}{ll}
((1+{\bf p})^k -1)\cdot L^{\{p\}}(1- k,\,\omega^{-k}) & \textrm{if $\xi\chi\omega^a$ is trivial}, \\[2mm]
L^{\{p\}}(1- k,\,\xi\chi\omega^{a-k}) & \textrm{otherwise}. 
\end{array}
\right.
\]
Accordingly, put 
\[
\mathcal{L}(\xi\chi\omega^a;\, X):=
\left\{
\begin{array}{ll}
X^{-1} \overline{\mathcal{L}}(\xi\chi\omega^a;\, X) & \textrm{if $\xi\chi\omega^a$ is trivial}, \\[2mm]
\overline{\mathcal{L}}(\xi\chi\omega^a;\, X) & \textrm{otherwise}. 
\end{array}
\right.
\]
We easily see that  
$\mathcal{L}(\xi\chi\omega^a;\, (1+{\bf p})^k -1)=L^{\{p\}}(1- k,\,\xi\chi\omega^{a-k})$ for each $k>1$. 
More generally, it turns out that if $\varepsilon : 1+{\bf p}\mathbb{Z}_p \to (\mathbb{Q}_p^{\rm alg})^{\times}$ is a character of finite order, 
then 
\begin{equation}
\mathcal{L}(\xi\chi\omega^a;\, \varepsilon(1+{\bf p})(1+{\bf p})^k -1)=L^{\{p\}}(1- k,\,\xi\chi\omega^{a-k}\varepsilon)
\end{equation}
for any $k>1$. 
On the other hand, for each $x \in 1+{\bf p}\mathbb{Z}_p$, put $s(x):=\log_p(x)/\log_p(1+{\bf p})$, where $\log_p
$ is the $p$-adic logarithm function in the sense of Iwasawa, and thus we have $s: 1+{\bf p}\mathbb{Z}_p \stackrel{\sim}{\to} \mathbb{Z}_p$. 
Then for each $T=\left[\begin{array}{c|c} 
T' & \\ \hline
& 0_{n-r}
\end{array}\right] \in {\rm Sym}_n^*(\mathbb{Z})$ \vspace*{1mm}with $T' \in {\rm Sym}_r^*(\mathbb{Z}) \cap {\rm GL}(r,\mathbb{Q})$ and $0 \le r \le n$, we define 
$\mathcal{A}_{\chi\omega^a}(T;\,X) \in F_{\Lambda}$ as follows:  
\begin{eqnarray}
\mathcal{A}_{\chi\omega^a}(T;\,X) 
&=& 2^{[(r+1)/2]-[(n+1)/2]} \prod_{i=[r/2]+1}^{[n/2]} \mathcal{L}(\chi\omega^{2a-2i};\,
(1+{\bf p})^{-2i}(1+X)^2 -1)  \\ 
&&\times 
\left\{
\begin{array}{ll}
\mathcal{L}(\left({\mathfrak{d}_{T'} \over *}\right) \chi\omega^{a-r/2};\,(1+{\bf p})^{-r/2}(1+X) -1) & \\[5mm]
\vspace*{5mm}
\hspace*{5mm}\times \displaystyle\prod_{\scriptstyle l \,\mid\, \mathfrak{f}_{T'}, \atop {\scriptstyle l \ne p}}  
F_l(T';\, \chi\omega^a(l)\, l^{-r-1}(1+X)^{s(\langle l \rangle)}) & \textrm{ if $r$ is even}, \\

 \displaystyle\prod_{\scriptstyle l \,\mid\, \mathfrak{D}_{T'}, \atop {\scriptstyle l \ne p} }  
F_l(T';\, \chi\omega^a(l)\, l^{-r-1}(1+X)^{s(\langle l \rangle)}) & \textrm{ if $r$ is odd}.  
\end{array}\right. 
\nonumber 
\end{eqnarray}
It follows from Theorem 4.2 (I\hspace{-.1em}I) that if $\kappa > n+1$, $\chi(-1)=(-1)^{\kappa}$ and further if $\kappa \equiv a \pmod{\varphi({\bf p})}$ (or equivalently, $\omega^{a-\kappa}$ is trivial), 
then  
\[
\mathcal{A}_{\chi\omega^a}(T;\,(1+{\bf p})^\kappa -1)=
A_{\kappa,\,\chi}^{*}(T). 
\]
Thus we obtain the former assertion. 
Moreover, it follows directly from Equation (14) that $\mathcal{B}^{(n)}(X)\cdot\mathcal{A}_{\chi\omega^a}(T;\,X) \in \Lambda$ for all $T$. 
Therefore $\overline{\mathcal{E}}_{\chi\omega^a}^{\,(n)}
(X)
\in \Lambda[[{\bf q}]]^{(n)}$. 
This completes the proof.  
\end{proof}

\section{$\Lambda$-adic Siegel Eisenstein series}

In this section, we fix an {\it odd}\, prime number $p$. We show that for a Dirichlet character $\chi$ modulo $M$ with $p \nmid M$ 
and an 
integer $a$ with $0 \le a < p-1$, 
the formal Fourier expansion $\mathcal{E}_{\chi\omega^a}^{(n)}(X)$ 
(resp. $\overline{\mathcal{E}}_{\chi\omega^a}^{\,(n)}(X)$) with coefficients in $F_{\Lambda}$ 
(resp. $\Lambda=\mathbb{Z}_p[\chi][[X]]$) defined in Theorem 4.4, give rise to classical Siegel modular forms via 
the specialization at $X=\varepsilon(1+p
)(1+p
)^{\kappa} -1$, where $\kappa$ is a positive integer sufficiently large 
and $\varepsilon : 1+p\mathbb{Z}_p \to (\mathbb{Q}_p^{\rm alg})^\times$ is a character of finite order. 

\medskip

First, from Lemma 2.1\,(I\hspace{-.1em}I), we deduce the following: 

\begin{thm}
If $\kappa$ is a 
positive integer with $\kappa>n+1
$ 
and further if $\varepsilon : 1+p\mathbb{Z}_p \to (\mathbb{Q}_p^{\rm alg})^\times$ is a character of exact order $p^m$ for some nonnegative integer $m$,  
then 
\[
\mathcal{E}_{\chi\omega^a}^{(n)}
(\varepsilon(1+p
)(1+p
)^{\kappa} -1)
=E_{\kappa,\,\chi\omega^{a-\kappa}\varepsilon}^{(n)} \in \mathscr{M}_\kappa(\Gamma_0(Mp^{m+1}
), \chi\omega^{a-\kappa}\varepsilon)^{(n)} 
\]
as long as $\omega^{2a-2\kappa}\varepsilon^2$ is non-trivial. 
\end{thm}

\begin{proof}For each $T=\left[\begin{array}{c|c} 
T' & \\ \hline
& 0_{n-r}
\end{array}\right] \in {\rm Sym}_n^*(\mathbb{Z})$ \vspace*{1mm}with $T' \in {\rm Sym}_r^*(\mathbb{Z}) \cap {\rm GL}(r,\mathbb{Q})$ and $0 \le r \le n$, 
the equation (14) is specialized as 
\begin{eqnarray}
\mathcal{A}_{\chi\omega^a}(T;\,\varepsilon(1+p
) (1+p
)^{\kappa}-1)
&=& 2^{[(r+1)/2]-[(n+1)/2]} \prod_{i=[r/2]+1}^{[n/2]} L^{\{p\}}(1-2\kappa+2i, \chi\omega^{2a-2\kappa}\varepsilon^2) \nonumber \\[2mm] 
&& \times 
\left\{
\begin{array}{ll}
L^{\{p\}}(1-\kappa+r/2, \, \left({\mathfrak{d}_{T'} \over *}\right) \chi\omega^{a-\kappa}\varepsilon) 
 & \\[5mm]
\hspace*{3mm}\times \displaystyle\prod_{\scriptstyle l \,\mid\, \mathfrak{f}_{T'}, \atop {\scriptstyle l \ne p}}  
F_l(T';\, \chi\omega^{a-\kappa}\varepsilon(l)\,l^{\kappa-r-1}) & \textrm{ if $r$ is even}, \vspace*{4mm} \\
 \displaystyle\prod_{\scriptstyle l \,\mid\, \mathfrak{D}_{T'}, \atop {\scriptstyle l \ne p} }  
F_l(T';\, \chi\omega^{a-\kappa}\varepsilon(l)\, l^{\kappa-r-1}) & \textrm{ if $r$ is odd}.  
\end{array}\right. 
\nonumber
\end{eqnarray}
Thus, if $\omega^{2a-2\kappa}\varepsilon^2$ is non-trivial, then Lemma 2.1 (I\hspace{-.1em}I) yields  
\[
\mathcal{A}_{\chi\omega^a}(T;\,\varepsilon(1+{\bf p}) (1+{\bf p})^{\kappa}-1)=A_{\kappa,\,\chi\omega^{a-\kappa}\varepsilon}(T). 
\]
This completes the proof. 
\end{proof}

On the other hand, by exploiting the vertical control theorem for $p$-adic Siegel modular forms in the sense of Hida \cite{Hid02}, 
we may deduce a slight weaker version of the preceding theorem as follows: 

\begin{thm}
If $\kappa$ is a positive integer with ${\kappa>
n(n+1)/2}$,  
and further if $\varepsilon : 1+
p\mathbb{Z}_p \to (\mathbb{Q}_p^{\rm alg})^\times$ is a 
character of exact order $p^m$ for some nonnegative 
integer $m$, 
then 
\[
\mathcal{E}_{\chi\omega^a}^{(n)}(\varepsilon(1+
p)(1+
p)^{\kappa} -1) \in \mathscr{M}_{\kappa}(\Gamma_0(M
p^{m+1}),\chi\omega^{a-\kappa}\varepsilon)^{(n)}
 \]
as long as $\omega^{a-\kappa}\varepsilon$ is non-trivial. 
\end{thm}

\begin{proof}
It follows from Theorem 4.4 that 
$\mathcal{E}_{\chi\omega^a}^{(n)}((1+
p)^{\kappa'} -1) \in \mathscr{M}_{\kappa'}(\Gamma_0(Mp),\,\chi)^{(n)}$
for infinitely many integers $\kappa'$ with $\kappa' >n+1$, $\chi(-1)=(-1)^{\kappa'}$ and $\kappa' \equiv a \pmod{p-1}$. Thus, Th\'eor\`eme 1.1 in \cite{Pi11} (generalizing 
\cite{Hid02} in more general settings) turns out that 
if an integer $\kappa
$ is sufficiently large, 
then the specialization of $\mathcal{E}_{\chi\omega^a}^{(n)}(X)$ at $X=\varepsilon(1+
p)(1+
p)^{\kappa} -1$ 
gives rise to an overconvergent $p$-adic Siegel modular form of weight $\kappa$, tame level $M$ and Iwahori level $p^{m+1}$ with 
character $\chi\omega^{a-\kappa}\varepsilon$ (for the precise definition, see \cite{S-U06,Pi11}). 
In addition, the explicit formula for Fourier coefficients of 
$\mathcal{E}_{\chi\omega^a}^{(n)}(\varepsilon(1+p
) (1+p
)^{\kappa}-1)$ also yields that it is an eigenform of $U_{p,n}$ whose eigenvalue 
is $1$. 
Therefore, if $\kappa$ satisfies the condition $\kappa>n(n+1)/2$ (cf. Hypoth\`ese 4.5.1 in \cite{B-P-S16}
), then the above-mentioned overconvergent form is indeed classical, that is, 
\[
\mathcal{E}_{\chi\omega^a}^{(n)}(\varepsilon(1+
p)(1+
p)^{\kappa} -1) \in \mathscr{M}_{\kappa}(\Gamma_0(M
p^{m+1}),\chi\omega^{a-\kappa}\varepsilon)^{(n)}\] 
(cf. Th\'eor\`eme 5.3.1 in [ibid]
). Thus we obtain the assertion. 
\end{proof}

\begin{rem}
Although Theorem 4.4 remains valid even for $p=2$, 
we may not deduce 
similar statements to Theorems 5.1 and 5.2 in the case where $p=2$, 
since there are some technical difficulties in understanding $\mathcal{E}_{\chi\omega^a}^{(n)}(\varepsilon(1+{\bf p})(1+{\bf p})^k -1)$ from automorphic and geometric viewpoints. 
\end{rem}

Now, as a summary 
of Theorems 4.4, 5.1 and 5.2 above, we have the following statement by which $\overline{\mathcal{E}}_{\chi\omega^a}^{\,(n)}(X)$ is said to be  
the {\it semi-ordinary $\Lambda$-adic Siegel Eisenseiten series} of genus $n$ and level $Mp^{\infty}$ associated with character $\chi\omega^a$: 

\begin{thm}
Let $M$ be an odd integer and $\chi$ a Dirichlet character modulo $M$ taken as follows: 
\begin{itemize}
\item[(i)] if $M=1$, then $\chi$ is trivial (or principal); 
\item[(ii)] if $M>1$, then $\chi$ is primitive and $\chi^2$ is locally non-trivial at every prime $l \mid M$; 
\end{itemize} 
For each integer $a$ with $0 \le a < p-1$, 
there exists a formal Fourier expansion \ 
$\overline{\mathcal{E}}_{\chi\omega^a}^{\,(n)}(X) \in \Lambda[[{\bf q}]]^{(n)}$ such that 
if $\kappa > n(n+1)/2$ and further if $\varepsilon : 1+p\mathbb{Z}_p \to (\mathbb{Q}_p^{\rm alg})^\times$ 
is a character of exact order $p^m$ with some $m \ge 0$, then 
\[
\overline{\mathcal{E}}_{\chi\omega^a}^{\,(n)}(\varepsilon(1+
p)(1+
p)^{\kappa} -1) \in \mathscr{M}_{\kappa}(\Gamma_0(M
p^{m+1}),\chi\omega^{a-\kappa}\varepsilon)^{(n)}
\]
is a Hecke eigenform whose eigenvalue of $U_{p,n}$ is $1$. 
In particular, 
\[
\overline{\mathcal{E}}_{\chi\omega^a}^{\,(n)}(\varepsilon(1+
p)(1+
p)^{\kappa} -1) = C_{\kappa,\,\varepsilon}^{(n)} \times \left\{
\begin{array}{ll}
(E_{\kappa,\,\chi}^{(n)})^{*} & \textrm{if $\omega^{a-\kappa}\varepsilon$ is trivial}, \\[2mm]
E_{\kappa,\,\chi\omega^{a-\kappa}\varepsilon}^{(n)} & \textrm{if $\omega^{2a-2\kappa}\varepsilon^2$ is non-trivial},  
\end{array}
\right.
\]
where $C_{\kappa,\,\varepsilon}^{(n)}=\mathcal{B}^{(n)}(\varepsilon(1+
p)(1+
p)^{\kappa} -1)$, 
that is, 
\begin{eqnarray*}
C_{\kappa,\,\varepsilon}^{(n)}
&=&
\prod_{i=1}^{[n/2]} \left\{\varepsilon^2(1+p)(1+p)^{2\kappa-2i} -1\right\}
\prod_{j=0}^{[n/2]} \left\{\varepsilon(1+p)(1+p)^{\kappa-j} -1\right\} \\
&=& 
\left\{\varepsilon(1+p)(1+p)^{\kappa} -1\right\} \prod_{i=1}^{[n/2]} \left\{\varepsilon(1+p)(1+p)^{\kappa-i} -1\right\}^2
\cdot \left\{\varepsilon(1+p)(1+p)^{\kappa-i} +1\right\}.
\end{eqnarray*}
\end{thm}

\begin{rem}
The preceding theorem for $n > 1$ can be regarded as a satisfactory generalization of Fact 1.1 to some extent. 
The only exception is 
our lack of knowledge about the specialization $\overline{\mathcal{E}}_{\chi\omega^a}^{\,(n)}(\varepsilon(1+
p)(1+
p)^{\kappa} -1)$  
in the case where $\omega^{a-\kappa}\varepsilon$ is non-trivial, but $\omega^{2a-2\kappa}\varepsilon^2$ is trivial. 
As shown in Theorem 5.2, it has been proved that if 
$\omega^{a-\kappa}\varepsilon$ is non-trivial, regardless of whether $\omega^{2a-2\kappa}\varepsilon^2$ is trivial or not,
$\overline{\mathcal{E}}_{\chi\omega^a}^{\,(n)}(\varepsilon(1+
p)(1+
p)^{\kappa} -1) \in \mathscr{M}_{\kappa}(\Gamma_0(M
p^{m+1}),\chi\omega^{a-\kappa}\varepsilon)^{(n)}$ is a Hecke eigenform which is semi-ordinary at $p$. 
\end{rem}




\begin{thebibliography}{9999999}
\bibitem[AZ]
{A-Z95}A.N.\,Andrianov and V.G.\,Zhuravl{\"e}v, 
\textit{Modular forms and Hecke operators}, 
Transl. of Math. Monogr., vol.145, Amer. Math. Soc., Providence, RI, 1995.
\bibitem[AL]
{A-L70}A.O.L.\,Atkin and J.\,Lehner, 
\textit{Hecke operators on $\Gamma_0(m)$}, 
Math. Ann. {\bf 185} (1970), 134--160. 
\bibitem[BPS]
{B-P-S16}S.\,Bijakowski, V.\,Pilloni and B.\,Stroh, 
\textit{Classicit\'e de formes modulaires surconvergentes}, 
Ann. of Math. {\bf 183} (2016), no.3, 975--1014.
\bibitem[B{\"o}]
{Boe05}S.\,B{\"o}cherer, 
\textit{On the Hecke operator U(p). With an appendix by Ralf Schmidt}, 
J. Math. Kyoto Univ. {\bf 45} (2005), no.4, 807--829. 
\bibitem[BS]
{B-S87}S.\,B{\"o}cherer and F.\,Sato, 
\textit{Rationality of certain formal power series related to local densities}, 
Comment. Math. Univ. St.Paul. {\bf 36} (1987), no.1, 53--86.
\bibitem[CP]
{C-P04}M.\,Courtieu and A.A.\,Panchishkin, 
\textit{Non-Archimedean L-Functions and Arithmetical Siegel Modular Forms}, 
Lecture Notes in Math., vol.1471, Springer-Verlag, Berlin, 2004.  
\bibitem[DR]
{D-R80}P.\,Deligne and K.A.\,Ribet, 
\textit{Values of abelian L-functions at negative integers over totally real fields}, 
Invent. Math. {\bf 59} (1980), 227--286.
%
\bibitem[Fe]
{Fei86}P.\,Feit, 
\textit{Poles and residues of Eisenstein series for symplectic and unitary groups}, 
Mem. Amer. Math. Soc. {\bf 61}, no.346, 1986. 
\bibitem[Fr]
{Fre83}E.\,Freitag, 
\textit{Siegelsche Modulfunktionen}, 
Grundlehren Math. Wiss., vol.254, Springer-Verlag, Berlin, 1983.
%
\bibitem[H1]
{Hid93}H.\,Hida, 
\textit{Elementary theory of $L$-functions and Eisenstein series}, 
London Math. Soc. Student Text, vol.26, Cambridge University Press, Cambridge, 1993. 
\bibitem[H2]
{Hid02}H.\,Hida, 
\textit{Control theorems for coherent sheaves on Shimura varieties of PEL-type}, 
J. Inst. Math. Jussieu {\bf 1} (2002), no.1, 1--76. 
\bibitem[H3]
{Hid04}H.\,Hida, 
\textit{$p$-adic automorphic forms on Shimura varieties}, 
Springer Monographs in Math., Springer-Verlag, New York, 2004.
\bibitem[Ik]
{Ike17}
T.\,Ikeda, 
\textit{On the functional equation of the Siegel series}, 
J. Number Theory {\bf 172} (2017), 44--62. 
\bibitem[K1]
{Kat99}H.\,Katsurada, 
\textit{An explicit formula for Siegel series}, 
Amer. J. Math. {\bf 121} (1999), no.2, 415--452. 
\bibitem[K2]
{Kat01}H.\,Katsurada, 
\textit{Euler factor of a certain Dirichlet series attached to Siegel Eisenstein series},  
Abh. Math. Sem. Univ. Hamburg {\bf 71} (2001), 81--90. 
%
\bibitem[Ki1]
{Kit84}Y.\,Kitaoka, 
\textit{Dirichlet series in the theory of quadratic forms}, 
Nagoya Math. J. {\bf 95} (1984), 73--84.
\bibitem[Ki2]
{Kit86}Y.\,Kitaoka, 
\textit{Local densities of quadratic forms and Fourier coefficients of Eisenstein series}, 
Nagoya Math. J. {\bf 103} (1986), 149--160.
\bibitem[Mi]
{Miy06}T.\,Miyake, 
\textit{Modular forms}, 
Springer Monographs in Math., Springer-Verlag, Berlin, 2006.
\bibitem[Pa]
{Pan00}A.A.\,Panchishkin, 
\textit{On the Siegel-Eisenstein measure and its applications}, 
Israel J. Math. {\bf 120} (2000), part B, 467--509. 
\bibitem[Pi]
{Pi11}V.\,Pilloni, 
\textit{Sur la th\'eorie de Hida pour le groupe ${\rm GSp}_{2g}$}, 
Bull. Soc. Math. France {\bf 140} (2012), no.3, 335--400. 
\bibitem[Se]
{Ser73}J.-P.\,Serre, 
\textit{Formes modulaires et functions z\^eta $p$-adiques}, 
{\it Modular functions of one variable, Vol.\,III} (Proc. Internat. Summer School, Univ. Antwerp, 1972), 191--268, Lecture Notes in Math., vol.350, Springer-Verlag, Berlin, 1973. 
\bibitem[S1]
{Shim82}G.\,Shimura, 
\textit{Confluent hypergeometric functions on tube domains}, Math. Ann. {\bf 260} (1982), no.3, 269--302.
\bibitem[S2]
{Shim94b}G.\,Shimura, 
\textit{Euler products and Fourier coefficients of automorphic forms on symplectic groups}, 
Invent. Math. {\bf 116} (1994), no.1--3, 531--576. 
\bibitem[SU]
{S-U06}C.\,Skinner and E.\,Urban, 
\textit{Sur les d{\'e}formations $p$-adiques de certaines repr{\'e}sentations automorphes}, 
J. Inst. Math. Jussieu {\bf 5} (2006), no.4, 626--698. 
\bibitem[Sw]
{Swe95}W.J.\,Sweet Jr., 
\textit{A computation of the gamma matrix of a family of $p$-adic zeta integrals}, 
J. Number Theory {\bf 55} (1995), no.2, 222--260.
\bibitem[T]
{Tak15}S.\,Takemori, 
\textit{Siegel Eisenstein series of degree $n$ and $\Lambda$-adic Eisenstein series}, 
J. Number Theory {\bf 149} (2015), 105--138. 
\bibitem[Ta]
{Tay88}R.L.\,Taylor, 
\textit{Congruences between modular forms}, 
Ph.D thesis, Princeton University, 1988. 
\bibitem[W]
{Wil88}A.\,Wiles, 
\textit{On ordinary $\lambda$-adic representations associated to modular forms}, 
Invent. Math. {\bf 94} (1988), no.3, 529--573. 
\end{thebibliography}
\end{document}